\documentclass{article}
\usepackage{amsmath, amssymb}
\usepackage{amsthm}

\newtheorem{theorem}{Theorem}[section]
\newtheorem{lemma}[theorem]{Lemma}
\newtheorem{proposition}[theorem]{Proposition}
\newtheorem{corollary}[theorem]{Corollary}
\newtheorem{example}[theorem]{Example}

\theoremstyle{definition}
\newtheorem{defn}[theorem]{Definition}

\newtheorem{algo}[theorem]{Algorithm}

\def\disc{\textrm{disc}}
\def\F{\mathbb{F}}
\def\C{\mathcal{C}}
\def\K{\mathbb{K}}
\def\L{\mathbb{L}}
\def\Q{\mathbb{Q}}
\def\E{\mathbb{E}}
\def\R{\mathbb{R}}
\def\spec{\mathrm{Spec}}

\begin{document}
\title{Fields of Parametrization and Optimal Affine Reparametrization of
Rational Curves}
\author{Luis Felipe Tabera\footnote{The author is supported by the contract
MTM2005-08690-C02-02}}
\maketitle
\begin{abstract}
In this paper we present three related results on the subject of fields of
parametrization. Let $\mathcal{C}$ be a rational curve over a field of
characteristic zero. Let $\mathbb{K}$ be a field finitely generated over
$\mathbb{Q}$, such that it is a field of definition of $\mathcal{C}$ but not a
field of parametrization. It is known that there are quadratic extensions of
$\mathbb{K}$ that parametrize $\mathcal{C}$. First, we prove that there are
infinitely many quadratic extensions of $\mathbb{K}$ that are fields of
parametrization of $\mathcal{C}$.

As a consequence, we prove that the witness variety, that appear in the context
of the parametric Weil's descente method, is always a special curve related to
algebraic extensions, called hypercircle. It is possible that the witness
variety is not a hypercircle for the given extension, but for an alternative
one.

We use these two facts to present an algorithm to solve the following optimal
reparametrization problem. Given a birational parametrization $\phi(t)$ of a
curve $\mathcal{C}$, compute the affine reparametrization $t\rightarrow at+b$
such $\phi(at+b)$ has coefficients over a field as small as possible. The main
advantage of this algorithm is that it does not need to compute any rational
point on the curve.
\end{abstract}

\section{Introduction}
Let $\K$ be a characteristic zero field, $\F$ its algebraic closure. Let $\C
\subseteq \mathbb{F}^N$ be a rational curve. We say that $\K$ is a \textit{field
of definition} of $\C$ if $\C$ can be described as the zero set of a system of
polynomials with coefficients in $\K$. Analogously, we say that $\C$ is
\textit{parametrizable} over $\K$ if there is a parametrization of $\C$ defined
by rational functions in $\K(t)$. Given any rational curve $\C$ there is a well
defined notion of field of definition $k$ of $\C$ which is the smallest field of
definition under inclusion. However, if $\C$ is not parametrizable over its
smallest field of definition, the best we can say is that there are
minimal fields of parametrization that are quadratic algebraic extensions of $k$
\cite{CHEV, SW1, Hilbert-Hurwitz}.

In this article we study how many minimal fields of parametrization $\C$ has.
The answer is trivial if $\C$ is parametrizable over its smallest field of
definition, but it is more subtle if $\C$ is not parametrizable over $k$. We
prove, using elementary tools from commutative algebra, that there are
infinitely many quadratic algebraic extensions $k(\gamma)$ of $k$ that
parametrizes $\C$.

This result is useful in the context of algebraically optimal reparametrization
of a rational curve $\C$. Suppose that $\C$ is a rational curve defined by a
birational parametrization $\phi(t)$ with coefficients over $\K(\alpha)$, where
$\alpha$ is algebraic of degree $n$ over $\K$. The problems of deciding if $\C$
can be parametrized over $\K$ and computing such a parametrization have been
studied (among others) in \cite{ARS-1,ARS-2, Ultraquadrics, Hyperc_to_units,
SV-Quasipolynomial, SW1}. Any birational parametrization of $\C$ with
coefficients in $\K$ is of the form $\phi(\frac{at+b}{ct+d})$, where
$\frac{at+b}{ct+d}\in \K(\alpha)(t)$. Computing a valid linear fraction
$\frac{at+b}{ct+d}$ that reparametrices $\phi$ over $\K$ is at least as hard as
computing a point in $\C \cap \K^N$. A parametric version of Weil's descente
method is proposed in \cite{ARS-2} to transform the curve $\C$ to a witness
variety where the problem is thought to be easier. It is known
\cite{ARS-2,Ultraquadrics} that $\C$ is parametrizable over $\K$ if and only if
the witness variety is a special curve called $\alpha$-hypercircle. These
$\alpha$-hypercircles have many interesting geometric properties
\cite{Generalizing-circles}, some of them related with the algebraic extension
$\K\subseteq \K(\alpha)$.

An $\alpha$-hypercircle is a curve that depends on $\K$ and $\alpha$. In this
article we prove that, if $\C$ is defined over $\K$, then the associated witness
variety to $\phi$ is a curve that is always a hypercircle for $\K(\beta)$ and
$\alpha$, where $\beta$ is algebraic of degree at most $2$ over $\K$ and
$[\K(\beta, \alpha):\K(\beta)]=[\K(\alpha) :\K]=n$. As a consequence we can
derive the following property of hypercircles: let $r$ be the geometric degree
of the hypercircle associated to $\phi$. Then there exists an intermediate field
$\K\subseteq \K(\gamma)\subseteq \K(\alpha)$, $[\K(\gamma) :\K]=r$, and an
affine change of variable $at+b$ such that $\phi(at+b)\in \K(\gamma)(t)$.
Moreover, this field $\K(\gamma)$ is of minimal degree among the fields
obtained by an affine change of parameter $\phi(at+b)$.
\[[\K(\gamma):\K]\leq [\K(\textrm{coeffs}(\phi(at+b))):\K]\]
The main advantage is that computing $\gamma$, $at+b$ and points in $\C\cap
\K(\gamma)^N$ is relatively easy compared to the problem of finding points in
$\C\cap \K^N$.

\subsection{Preliminaries}\label{sec:prelim}

Let $\mathbb{K}$ be a characteristic zero field, $\F$ its algebraic closure. Let
$\C\subseteq \F^N$ be an irreducible algebraic variety over $\mathbb{F}$. We say
that $\mathbb{K}$ is a \textit{field of definition} of $\C$ if $\C$ can be
expressed as the set of common zeros in $\F^N$ of an ideal in $\K[x_1, \ldots,
x_N]$.

From an algebraic point of view, consider the canonical ring inclusions
$\K\subseteq \mathbb{K}[x_1,\ldots, x_N]\subseteq \mathbb{F}[x_1,\ldots, x_N]$.
Let $P\in \spec (\F[x_1,\ldots,x_N])$ be a prime ideal. Denote by
\[\K(P)=Frac\frac{\K[x_1,\ldots,x_N]}{P^c}\] where $P^c$ is the contracted ideal
$P^c=P\cap \K[x_1,\ldots,x_n]$. Consider the induced field extensions
\[\K\subseteq \K(P)\subseteq \F(P)\]

Recall that a finitely generated extension of fields $A\subseteq B$ is
\textit{regular} if the extension is separable and $A$ is algebraically closed
in $B$ (cf. \cite{ZS-2}). Also, an ideal $P\in\spec(\mathbb{K}[x_1,\ldots,x_N])$
is called \textit{absolutely prime} if $P^e$ is prime in $\L[x_1,\ldots,x_N]$
for any extension $\L$ of $\K$.

\begin{proposition}\label{prop:def_field}
The following statements are equivalent:
\begin{enumerate}
\item $P$ has a set of generators with coefficients in $\K$.
\item $P=P^{ce}$ is an extended ideal for $\K[x_1,\ldots,x_N]\subseteq
\F[x_1,\ldots,x_N]$.
\item $P^{ce}$ is prime in $\F[x_1,\ldots,x_N]$.
\item $P^c$ is absolutely prime.
\item The field extension $\K \subseteq \K(P)$ is regular.
\end{enumerate}
\end{proposition}
\begin{proof}
It follows easily form the results in \cite{ZS-2} Ch VII \S 11.
\end{proof}

If these properties hold, we say that $\K$ is \textit{a field of definition} of
$P$. Obviously, $\F$ is always a field of definition of any $P$ and any
extension of a field of definition is again a field of definition. It is known
that there is a minimum field of definition $k$ of $P$ with respect to
inclusion. This field $k$ is \textit{the} field of definition of $P$. This field
is always finitely generated over $\Q$.

\begin{defn}\label{defn:K-birational}
Let $P \in \spec(\F[x_1,\ldots, x_N])$, $Q\in \spec(\F[y_1,\ldots, y_M])$ be two
prime ideals. A \textit{$\K$-isomorphism} of $\F(P)$ and $\F(Q)$ is an
isomorphism of fields $\phi$ such that $\phi|_\F=id$, $\phi([x_i])\in \K(Q)$,
$\phi^{-1}([y_j])\in \K(P)$, $1\leq i\leq N$, $1\leq j\leq M$. In this case we
say that the varieties defined by $P$ and $Q$ are \textit{$\K$-birational}.
\end{defn}

\begin{defn}\label{defn:field_of_parametrization}
Let $P$ define a rational curve. That is \[\F(P) \cong \F(t)\] We say that $\K$
is \textit{a field of parametrization} of $P$ if there is a $\K$-isomorphism
between $\F(P)$ and $\F(t)$. This isomorphism is called a
\textit{$\K$-parametrization}. We say in this case that the curve is
\textit{$\K$-parametrizable}.
\end{defn}

If $P$ is $\K$-parametrizable, then it follows that $\K(P) \cong \K(t)$ by an
isomorphism that is the identity in $\K$. Then, since $\K\subseteq \K(t)$ is
regular, by Proposition~\ref{prop:def_field}, $\K$ is a field of definition.

But it could happen that $\K$ is a field of definition of $P$ without being a
field of parametrization. Then, it is known \cite{CHEV} that there exists at
least an element $\eta$ that is quadratic over $\K$ such that $\K(\eta)$ is a
field of parametrization of $P$. So, if the field of definition $k$ of $P$ is
not a field of parametrization, there are always minimal fields of
parametrization under inclusion that are quadratic extensions of $k$.

\begin{example}\label{example:varios_casos}
Let $P=(x+iy), Q=(x^2+y^2+1) \subseteq \mathbb{C}[x,y]$, $\K=\mathbb{R}$. Then
$P^c=(x^2+y^2)$, the extension $\mathbb{R} \subseteq \R(P)$
is not regular, because $(x/y)^2+1=0$, $x/y$ is algebraic over $\mathbb{R}$. In
fact
\[\begin{matrix}\phi:&\textrm{Frac\ } \mathbb{R}[x,y]/(x^2+y^2)& \rightarrow
&\mathbb{C}(x)\\
  &[x]&\mapsto&x\\
  &[y]&\mapsto&ix
  \end{matrix}
\]
is an isomorphism of fields. So $\mathbb{R}$ is not a field of definition of
$P$. On the other hand, $\mathbb{R}$ is a field of definition of $Q$ but
$\mathbb{R}(Q)$ is not a real field, so it is not isomorphic to $\mathbb{R}(t)$
and $\mathbb{R}$ is not a field of parametrization of $Q$ (cf. \cite{RS-2}). If
we look for the fields of parametrization of $Q$ that contains $\mathbb{R}$, the
only possibility is the complex field. This is not a contradiction with the aim
to prove that there are infinitely many minimal fields of parametrization,
because $\mathbb{R}$ is not \textbf{the} field of definition of $Q$. We will
prove that there are infinitely many quadratic fields over $\Q$ that
parametrices $Q$.
\end{example}

\begin{proposition}
Let $P\in \spec(\F[x_1,\ldots, x_N])$ define a curve, let $\K$ be a field of
definition of $P$ then $P$ is $\K$-birational to a plane curve $Q$ defined over
$\K$.
\end{proposition}
\begin{proof}
We write the classical proof for completeness. Since $P$ defines a curve,
$\K(P)$ is of transcendence degree 1 over $\K$. Let $x=g_1([x_1],\ldots,
[x_N])\in P$ be transcendent over $\K$. Then $\K(x) \subseteq \K(P)$ is algebraic
and finite, so there is an element $y=g_2([x_1],\ldots, [x_N])$ algebraic over
$\K(x)$ such that $\K(x,y)=\K(P)$. Taking the minimal polynomial of $y$ over
$\K(x)$ and clearing denominators, we get an equation $f(x,y)$ such that $Frac
\K[x,y]/(f(x,y))\cong \K(P)$. Let $Q$ be the curve defined by $f(x,y)$. Clearly
$Q$ is defined over $\K$ and the isomorphism
\[\begin{matrix}\phi: &\F(Q)&\rightarrow &\F(P)\\
  &[x]&\mapsto&g_1([x_1],\ldots, [x_N])\\
  &[y]&\mapsto&g_2([x_1],\ldots, [x_N])\\
  \end{matrix}
\]
is a $\K$-isomorphism.
\end{proof}

\begin{theorem}[Hilbert-Hurwitz]\label{teo:Hilbert-Hurwitz}
Let $P=(f(x,y))\in \spec(\F[x,y])$ define a rational curve defined over $\K$,
$\deg(f)=d>2$, $f(x,y)\in \K[x,y]$. Then the curve defined by $P$ is
$\K$-birational to a curve of degree $d-2$ defined over $\K$.
\end{theorem}
\begin{proof}
\cite{Hilbert-Hurwitz,SW1}
\end{proof}

\section{The number of fields of parametrization}\label{sec:number_of_fileds}
By the previous results, every rational curve $\C$ defined over $\K$ is
$\K$-birational to a plane curve $\C'$ of degree at most 2. If $\C'$ is a line
or a conic that is $\K$-parametrizable, then $\C$ is $\K$-parametrizable. More
generally, if $\L$ is a field of parametrization of $C'$ and $\K\subseteq \L$
then $\L$ is a field of parametrization of $\C$. Using standard projective
transformations, we can reduce the problem of deciding if $\C$ is parametrizable
over $\K$ to a plane conic defined by a polynomial of the form $(ax^2+by^2+c)$.
In the general case, one would use the primitive element theorem to find a plane
curve birational to the given curve and then use Hilbert-Hurwitz to compute a
conic birational to it. It is worthy to comment that we may change the field of
definition of the curve during this process. For example, the curve defined by
$(x-\sqrt{2})\subseteq \Q(\sqrt{2})[x,y]$ is defined over $\Q(\sqrt{2})$ and it
is $\Q(\sqrt{2})$-birational with $(0)\subseteq \Q[x]$. In this case the field
of definition of the second variety is $\Q$ and not $\Q(\sqrt{2})$. So, we have
to be flexible with what fields of definition we are dealing with and not work
only with the field of definition of the original curve. But, in order to avoid
cases such as Example~\ref{example:varios_casos}, we cannot take too big fields.
To solve this problem, we will restrict to fields of definition $\K$ that are
finitely generated over $\Q$ (as a field) and are not fields of parametrization.
We may suppose, without loss of generality, that
$\mathbb{K}=\mathbb{Q}(y_1,\ldots,y_u)(\beta)$, where $y_1,\ldots, y_u$ are
indeterminates and $\beta$ is algebraic of degree $v$ over
$\mathbb{Q}(y_1,\ldots, y_u)$. Let $\mathbb{E}$ be the integral closure of the
ring $R=\mathbb{Z}[y_1,\ldots,y_u]$ in $\mathbb{K}$. Without loss of generality,
we may suppose that $\beta$ is integral over $R$ and that $\beta= \beta_1,
\ldots, \beta_v$ are the conjugates of $\beta$ in $\F$. First, we need some
elementary properties of these rings.

\begin{lemma}
$\mathbb{E}$ is a finitely generated $R$-module.
\end{lemma}
\begin{proof}
This is a well known result that follows from the fact that the discriminant $D$
of $\beta$ is a universal denominator for the ring of integers,
$D\mathbb{E}\subseteq <1,\beta,\ldots,\beta^{v-1}>$ (See
\cite{Integral-closure-of-noetherian-rings}). We know that $\{1, \beta, \beta^2,
\ldots, \beta^{v-1}\}$ is a base of the $\mathbb{Q}(y_1,\ldots,y_u)$-vectorial
space $\mathbb{K}$. Let $\mathbb{D}$ be the integral closure of $R$ in $\F$.
Take $a\in \mathbb{E}$, then, $a=a_0 + a_1\beta+\ldots +a_{v-1}\beta^{v-1}$,
$a_i\in\mathbb{Q}(y_1,\ldots,y_u)$. Let $a=b_1,\ldots,b_v$ be the conjugates of
$a$ in $\mathbb{D}$ defined by:
\[\begin{pmatrix} b_1\\ \vdots\\b_v\end{pmatrix}=\begin{pmatrix}
1&\beta&\ldots&\beta^{v-1}\\
\multicolumn{4}{c}{\cdots\cdots}\\
1&\beta_v&\ldots&\beta_v^{v-1}\\
\end{pmatrix}
\begin{pmatrix}a_0\\ \vdots\\ a_{v-1}\end{pmatrix}
\]
with coefficient matrix $B$. Multiplying by the adjoint matrix (whose elements
are in $\mathbb{D}$, that is, are integral elements) and $det(B)$ we obtain:
\[det(B)adj(B)\begin{pmatrix} b_1\\ \vdots\\b_v\end{pmatrix}=det(B)^2
\begin{pmatrix}a_0\\ \vdots\\ a_{v-1}\end{pmatrix}
\]
The left part of the equation belongs to $\mathbb{D}$. Note that
$D=det(B)^2=\disc(\beta)\in R$ belongs to the base ring. So, every $Da_i$ is
integral over $R$ and belongs to $\mathbb{Q}(y_1,\ldots,y_u)$. Since $R$ is a
UFD, it is integrally closed in its field of fractions, so $Da_i\in R$.

To sum up, if $a=a_0+a_1\beta+\ldots+a_{v-1}\beta^{v-1}\in\mathbb{E}$, then
$Da_i\in R, 0\leq i\leq v-1$. So $\mathbb{E}\subseteq
<1/D,\beta/D,\ldots,\beta^{v-1}/D>\subseteq\mathbb{K}$ a finitely generated
$R$-module. Since $R$ is noetherian, $\mathbb{E}$ is also finitely generated.
\end{proof}

Now, we define a multiplicative function in order to have some control on
divisibility in $\mathbb{E}$.

\begin{defn}\label{defn:norm}
Consider the map $N_1: \mathbb{Z}[y_1,\ldots, y_u]\rightarrow \mathbb{N}$
defined by:
\[N_1(x)=\left\{\begin{matrix}|x|& \textrm{\ if\ } x\in \mathbb{Z}\\
2& \textrm{\ if\ }x\notin \mathbb{Z}, x\textrm{\ irreducible}\\
2^k|c| &\textrm{\ if\ } x=cf_1\ldots f_k, c\in \mathbb{Z},
f_i\notin\mathbb{Z}\textrm{\ irreducible}
\end{matrix}\right.\]
$N_1$ is a multiplicative map such that $N_1(x)=0\leftrightarrow x=0$,
$N_1(x)=1$ if and only if $x\in\{1,-1\}$.

Define $N_2:\mathbb{E}\rightarrow \mathbb{Z}[y_1,\ldots,y_u]$ as the restriction
of the norm associated to the algebraic extension $\mathbb{Q}(y_1,\ldots,
y_u)\subseteq \mathbb{K}$. Let $N=N_1\circ N_2:\mathbb{E}\rightarrow
\mathbb{N}$.
\end{defn}

\begin{lemma}\label{lem:norma}In the previous conditions:
\begin{itemize}
\item $N(ab)=N(a)N(b)$, $N$ is multiplicative.
\item $N(x)=0$ if and only if $x=0$.
\item $N(x)=1$ if and only if $x$ is a unit in $\mathbb{E}$.
\item If $A,B\in \mathbb{E}^*$ and $N(A)\nmid N(B)$ then $A\nmid B$.
\end{itemize}
\end{lemma}
\begin{proof}
$N$ is the composition of multiplicative maps, so it is multiplicative. If
$x=0$, then it is clear that $N(x)=0$. If $N(x)=0$, then $N_1(N_2(x))=0$, so
$N_2(x)=0$ and $x=0$. If $x$ is a unit in $\mathbb{E}$ then
$1=N(1)=N(xx^{-1})=N(x)N(x^{-1})$ so $N(x)$ is a natural number that divides
$1$, it has to be $1$. Reciprocally, if $N(x)=1$, then $N_2(x)=\pm 1$. The
minimal polynomial of $x$ over $\mathbb{Z}[y_1, \ldots, y_u]$ is of the form
$x^s +a_1x^{s-1}+ \ldots +a_{s-1}x \pm 1$, $s\geq 1$. So, in $\mathbb{E}$,
$x(x^{s-1}+a_{s-1}x^{s-2}+\ldots+a_{s-1})=\mp 1$ and $x$ is a unit. The last
item follows from the multiplicativity of $N$.
\end{proof}

\begin{lemma}
Let $x\in \mathbb{E}^*$, let $p\in \mathbb{Z}$ be an integer prime such that
$p>N(x)$, then $x\notin \sqrt{p\mathbb{E}}$ and there is a maximal ideal $m_p$
in $\mathbb{E}$ such that $p\in m_p$, $x\notin m_p$.
\end{lemma}
\begin{proof}
If $p>N(x)$, $x\neq 0$, then $p\nmid N(x)$ in $\mathbb{Z}$, so $N(p)=p^v\nmid
N(x^K)$ and $p\nmid x^K$ in $\E$. Thus $x\notin \sqrt{p\mathbb{E}}$. In this
context, consider the ring of fractions $x^{-1}\mathbb{E}$ whose denominators
are the powers of $x$. $(p)^e$ is a proper ideal in $x^{-1}\mathbb{E}$. Let $P$
be a maximal ideal of $x^{-1}\mathbb{E}$. Its contraction is a prime ideal
containing $p$ but not $x^N$. Consider the map
\[\phi:\mathbb{E}/P^c\rightarrow x^{-1}\mathbb{E}/P\]
such that $\phi(a+P^c)=a/1+P$. $\phi$ is well defined and injective. Now, the
field of the right is of characteristic $p$ and is a finitely generated
$\mathbb{F}_p$-algebra, so it has to be an algebraic extension of
$\mathbb{F}_p$. It follows that $\mathbb{E}/P^c$ is a finite domain, so a field
and $P^c$ is a maximal ideal such that $p\in P^c$, $x\notin P^c$.
\end{proof}

\begin{lemma}\label{lem:residual_nonsquares}
Let $m$ be a maximal ideal in $\mathbb{E}$ such that char$(\mathbb{E}/m)=p\neq
2$, $p\equiv 1 \mod 4$. Let $e\not\equiv 0$ mod $m$. Then, there is a $n\in
\mathbb{E}$ such that $1+en^2$ is not a square mod $m$.
\end{lemma}
\begin{proof}
$\mathbb{E}/m\cong \mathbb{F}_{p^q}$, the field of $p^q$ elements. Suppose that
$1+en^2$ is always a square mod $m$. Then $[1+ex]$ is a bijection of
$\mathbb{F}_{p^q}$ such that maps the squares into themselves. It follows that
it is a bijection among the squares. In particular, there is a $[n]$ such that
$[1+en^2]=[0]$ a square. As $p\equiv 1\mod 4$, $[-1]$ is a square in
$\mathbb{F}_{p^q}$. So $[e]^{-1}=[-n^2]$ is a square and also $[e]$ is a square.
Thus, $[1+x]$ is a bijection on the set of squares. But this is impossible. If
this where the case, every element of $\mathbb{F}_p$ would be a square and $[x]$
is a square if and only if $[x+i]$ is a square for every $i\in \mathbb{F}_p$. In
$\mathbb{F}_{p^q}$ there are $\frac{p^q+1}{2}$ squares. The equivalence relation
$a\sim b\leftrightarrow a-b\in\mathbb{F}_p$ produces a partition of
$\mathbb{F}_{p^q}$ into equivalence classes, each one containing exactly $p$
elements. In a equivalence class, either every element is a square or none is.
Hence, it follows that the number of squares is divisible by $p$, while
$(p^r+1)/2$ is not.
\end{proof}

Note that depending on the maximal $m$, it may happen that necessarily $n\notin
\mathbb{Z}$ in the Lemma. In order to prove that $\C$ has infinitely many fields
of parametrization, we suppose without loss of generality that $\C$ is a plane
conic defined by the polynomial $ax^2+by^2+c$, where $a,b,c\in \mathbb{E}^*$.
Suppose that the conic does not have points with coefficients in $\mathbb{K}$.
Let $n\in \mathbb{E}$ be such that $a+bn^2\neq 0$. The intersection of the conic
and the line $y=nx$ provides the points
\[\left(\pm\sqrt{\frac{-c}{a+bn^2}},\pm n\sqrt{\frac{-c}{a+bn^2}}\right)\]
that have coefficients in $\mathbb{K}(\sqrt{\frac{-c}{a+bn^2}})$. This is a
field of parametrization of $\C$.

\begin{lemma}\label{lem:distinct_fields}
Let $n\neq m\in \mathbb{K}^*$, such than $a+bn^2\neq 0\neq a+bm^2$. Then,
$\mathbb{K}(\sqrt{\frac{-c}{a+bn^2}})=\mathbb{K}(\sqrt{\frac{-c}{a+bm^2}})$ if
and only if $\frac{a+bn^2}{a+bm^2}$ is a square in $\mathbb{K}$.
\end{lemma}
\begin{proof}
Straightforward.
% If $\frac{a+bn^2}{a+bm^2}=r^2$, then
% $\sqrt{\frac{-c}{a+bm^2}}=r\sqrt{\frac{-c}{a+bn^2}}$ or
% $\sqrt{\frac{-c}{a+bm^2}}=-r\sqrt{\frac{-c}{a+bn^2}}$. Suppose now that the
% fields are equal, then $\sqrt{\frac{-c}{a+bm^2}}=r+s\sqrt{\frac{-c}{a+bn^2}}$,
% $r,s\in\mathbb{K}$. Note that $s\neq 0$ because if $s=0$ then the conic would
% have solutions over $\mathbb{K}$. Squaring the expression, we obtain
% \[\frac{-c}{a+bm^2}=r^2+s^2\frac{-c}{a+bn^2}+2rs\sqrt{\frac{-c}{a+bn^2}}\]
% then, again $rs=0$, because if not, there would be $\mathbb{K}$-rational points
% in $\mathcal{C}$. But as $s\neq 0$, then $r=0$ and
% \[\frac{-c}{a+bm^2}=s^2\frac{-c}{a+bn^2}\]
% So $\frac{a+bn^2}{a+bm^2}=s^2$ is a square in $\mathbb{K}$.
\end{proof}

\begin{theorem}\label{teo:infinito-set}
In the previous conditions. There is an infinite set $S\subseteq \mathbb{E}$
such that for all $p\neq q\in S$, $a+bp^2\neq 0$, $a+bq^2\neq 0$ and
$\frac{a+bp^2}{a+bq^2}$ is not a square in $\mathbb{K}$.
\end{theorem}
\begin{proof}
Let $P$ be an enumeration of the integer primes $p\in \mathbb{Z}$ such that
$p\equiv 1\mod 4$, $a+bp^2\neq 0$, $p>N(ab)$. The set $P$ is infinite. To each
prime $p_i$ we associate to it a maximal ideal $m_i$ such that $p\in m_i$,
$ab\notin m_i$. It is clear that all these maximal ideals are different. We are
going to construct the set $S$ inductively from $P$. Define $S_1=\{p_1\}$,
suppose that we have defined a set $S_i=\{n_1,\ldots, n_i\}$ such that:
\begin{itemize}
\item $n_j\in m_{p_j}$, $1\leq j\leq i$
\item $a+bn_j^2\neq 0$, $1\leq j\leq i$
\item If $j\neq k$, $\frac{a+bn_j^2}{a+bn_k^2}$ is not a square in
$\mathbb{K}$. 
\end{itemize}
Then, we define $S_{i+1}$ as follows. For $1\leq j\leq i$, $[a+bn_j^2]=[a]$ mod
$m_j$, so $\frac{[a+bx^2]}{[a+bn_j^2]}=[1]+[b][a]^{-1}[x^2]$ mod $m_j$. By
Lemma~\ref{lem:residual_nonsquares}, there is a $q_j\in \E$ such that
$[1]+[b][a]^{-1}[q_j^2]$ is not a square mod $m_j$. Now, by the Chinese reminder
theorem, as the ideals $m_1,\ldots, m_{i+1}$ are pairwise comaximal, there is an
element $n_{i+1}$ such that $n_{i+1}\equiv q_j\mod m_j$, $1\leq j\leq i$;
$n_{i+1}\equiv 0\mod m_{i+1}$. Then $n_{i+1}\in m_{i+1}$. As
$[1]+[b][a^{-1}][n_{i+1}^2]$ is not a square mod $m_j$, $1\leq j\leq i$ then,
any representative of this class in $\E$ is not a square in $\E$. However, we
want to prove that $\frac{a+bn_{i+1}^2}{a+bn_j^2}$ is not a square in
$\mathbb{K}$ instead.

Consider the extension of rings $\E\subseteq \E'=(a+bn_j^2)^{-1} \mathbb{E}
\subseteq \K$. Since $\E$ is integrally closed in $\K$, then $\E'$ is also
integrally closed in $\K$. So, an element $x\in \E'$ is a square in $\E'$ if and
only if it is a square in $\K$. By construction, $\frac{a+bn_{i+1}^2}{a+bn_j^2}$
is not a square in $\E' /m_j\E' $, so it is not a square in $\E'$ and hence, it
is not a square in $\K$.

We define $S_{i+1}=\{n_1,\ldots,n_{i+1}\}$.
\end{proof}

\begin{corollary}
Let $\mathcal{C}=\{ax^2+by^2+c=0\}$, $a,b,c\in \mathbb{E}^*$ be an (absolutely)
irreducible conic that has no points with coordinates in $\mathbb{K}$, then
there are infinitely many distinct quadratic fields over $\mathbb{K}$ such that
the curve has points defined over them.
\end{corollary}
\begin{proof}
Let $S$ be the set defined in Theorem~\ref{teo:infinito-set}, by
Lemma~\ref{lem:distinct_fields} we have that the fields
\[\mathbb{K}\left(\sqrt{\frac{-c}{a+bn^2}}\right), n\in S\] are pairwise
distinct and $\mathcal{C}$ has points defined over them
\end{proof}

\begin{corollary}\label{cor:fields_of_parametrization}
Let $\mathbb{K}$ be a field of definition of a rational curve $\mathcal{C}$ such
that $\K$ is finitely generated over $\Q$. If $\mathcal{C}$ is not
parametrizable over $\mathbb{K}$, then $\mathcal{C}$ has infinitely many
distinct fields of parametrization that are quadratic over $\mathbb{K}$. In
particular, if $\C$ is not parametrizable over its minimum field of definition
$k$, then there are infinitely many minimal fields of parametrization that are
quadratic over $k$.
\end{corollary}

This Corollary is equivalent to the following fact. Let $\mathcal{C}$ be a genus
zero curve defined over $\K$. Suppose that $\mathcal{C}$ is not parametrizable
over $\K$. Then, for any field of parametrization $\K(\alpha)$, where $\alpha$
is algebraic over $\mathbb{K}$, there is a quadratic element $\gamma$ over
$\mathbb{K}$ such that $\mathbb{K}(\gamma)$ is a field of parametrization of
$\mathcal{C}$ and $\mathbb{K}(\alpha)\cap \mathbb{K}(\gamma)=\mathbb{K}$.

The main disadvantage of the method in Theorem~\ref{teo:infinito-set} is that it
is not constructive, because of the use of the maximal ideals $m_i$. We can give
a constructive proof of the theorem whenever our base ring $R$ is the ring of
integers $\mathbb{Z}$ and $\mathbb{K}=\mathbb{Q}(\alpha)$ is an algebraic number
field. In this case, the set $S$ can be taken as a set of integer primes.

\begin{theorem}\label{teo:conjunto_infinito_S}
Given $a,b\in \mathbb{Z}^*$, there is an infinite set $S$ such that: every
element in $S$ is a prime $p\equiv 1 \mod 4$, $p\nmid  ab$, $a+bp^2\neq 0$ and,
if $p,q\in S$, then \[f(p,q)=\frac{a+bp^2}{a+bq^2}\] is not a square in
$\mathbb{Q}$.
\end{theorem}
\begin{proof}
As in Theorem~\ref{teo:infinito-set}, we will define $S$ inductively, starting
from $S_1=\{p_1\}$ with $p_1$ any prime $p_1\equiv 1$ mod 4 such that $p_1\nmid
ab$. Note that $a+bp_1^2\neq 0$. Suppose we have defined a set
$S_i=\{p_1,\ldots,p_i\}$ such that $p_j$ is prime, $p_j\equiv 1$ mod 4,
$p_j\nmid  ab$, and if $k\neq j$ then $\frac{a+bp_k^2}{a+bp_j^2}$ is not a
square in $\mathbb{Q}$. We want to construct the set $S_{i+1}$. Consider the
residual polynomial in the variable $n$ with coefficients in $\mathbb{F}_{p_j}$
\[[a+bn^2][a+bp_j^2]^{-1} = [1]+[a]^{-1}[bn^2]=[1+en^2].\]
Note that $[e]\neq [0]$ is well defined because $p_j\nmid  ab$. Let $n_j$ be
such that such that $[1+en_j^2]$ is not a square. Let $p$ be a prime such that
$p\equiv 1 \mod 4$, $p\equiv n_j\mod p_j$, $(p,ab)=1$. This prime always exists:
from the Chinese remainder theorem, we can compute the unique class $N$ mod
$4p_1\cdots p_i$ such that verifies the previous modular equations. As $N$ is a
unit mod $4p_1\cdots p_i$ , we can apply Dirichlet's theorem and find a prime
$p\equiv N\mod 4p_1\cdots p_i$ such that, in addition, it does not divide $ab$.
Take $p_{i+1}=p$ in $S_{i+1}$. By construction, $p \equiv 1 \mod 4$, $p \nmid
ab$ and, in $\mathbb{Z}/p_j\mathbb{Z}$
\[[a+bp^2][a+bp_j^2]^{-1}=[a+bm_j^2][a+bp_j^2]^{-1}=[1+em_j^2]\]
which is not a square mod $p_j$, so $\frac{a+bp^2}{a+bp_j^2}$ is not a square in
$\mathbb{Q}$, $1\leq j\leq i$.
\end{proof}

Note that, in this case, the set obtained differs from the set obtained in
Theorem~\ref{teo:infinito-set}. Because this set is a set on integer primes,
while the set $S=\{q_i\}$ obtained using Theorem~\ref{teo:infinito-set} is such
that $q_i$ is a multiple if the $i$-th integer prime which is $\equiv 1$ mod 4.

\begin{example}
Let $\mathcal{C}=x^2+y^2-6=0$ which does not have points in $\mathbb{Q}^2$. We
look for a set of integers such that $\frac{1+n^2}{1+m^2}$ is never a square.
Take $p_1=5$. Now we search a $[n]$ such that $1+n^2$ is not a square mod $5$.
For example $1+1^2=2$ is not a square mod 5. Now compute a prime $p_2$ such
that:
\[p_2\equiv 1 \mod 4,\quad p_2\equiv 1 \mod 5\]
By the Chinese reminder Theorem. $p_2\equiv 1\mod 20$ and, we can take, for
example $p_2=41$. Now we need to compute $p_3$, we impose $1+n^2$ not to be a
square mod $41$, the first non square of this form is $1+4^2=17$. Again we have
the following system of residual equations
\[p_3\equiv 1 \mod 4,\ p_3\equiv\ 1 \mod 5,\ p_3\equiv\ 4 \mod 41\]
So, this time, $p_3\equiv 701 \mod 820$. we can take $p_3=701$. Applying again
this method, we arrive that the next prime must be $p_4\equiv 266381 \mod
574820$ and, in particular we can take, $p_4=266381$. That is, the intersection
of $\mathcal{C}$ with the lines $y=5x$, $y=41x$, $y=701x$ and $y=266381x$ gives
four different quadratic fields of parametrization of the conic. In this case,
we obtain:
\[\mathbb{Q}\Big(\sqrt{\frac{3}{13}}\Big), \mathbb{Q}(\sqrt{3}),
\mathbb{Q}\Big(\sqrt{\frac{3}{245701}}\Big),
\mathbb{Q}\Big(\sqrt{\frac{3}{35479418581}}\Big)\]

If we apply instead Theorem~\ref{teo:infinito-set} we will get that the first
terms of $S$ are:
\[S=\{5,26,391,4031,175306,9276086,\ldots\}\]
\end{example}

\section{Any witness variety is a hypercircle}\label{sec:witness}

Let $\C \subseteq \F^N$ be a rational curve defined by the birational
parametrization $(\phi_1(t), \ldots, \phi_N(t))$, where $\phi_i(t)\in
\K(\alpha)(t)$. Suppose that the rational functions are written with common
denominator $g$, $\phi_i(t)=f_i/g$, $1\leq i\leq N$, $\gcd(f_1,\ldots,f_N,g)=1$
and that $\alpha$ is algebraic of degree $n$ over $\K$.

Perform Weil's descente parametric version presented in \cite{ARS-2}, substitute
$t=\sum_{i=0}^{n-1} \alpha^it_i$, where $t_i$ are new variables. Rewrite:
\[\phi_j\left(\sum_{i=0}^{n-1} \alpha^it_i\right)= \sum_{i=0}^{n-1} \alpha^i
\psi_{ij}(t_0,\ldots,t_{n-1}), \psi_{ij}= \frac{F_{ij}}{\delta} \in
\K(t_0,\ldots, t_{n-1})\]
Let $\mathcal{Z}$ be the Zariski closure of \[\{F_{ij}=0\ |\ 1\leq i\leq n-1,\
1\leq j\leq N\}\setminus \{\delta=0\}\subseteq \F^n.\] $\mathcal{Z}$ is called
the \textit{witness variety} of the parametrization $\phi$.

\begin{theorem}\label{teo:witness_variety_properties}
With the previous notation:
\begin{itemize}
\item Dim $\mathcal{Z}\leq 1$.
\item $\mathcal{Z}$ has at most one 1-dimensional component.
\item $\K$ is a field of definition of $\C$ if and only if Dim $\mathcal{Z}=1$.
\item $\C$ is parametrizable over $\K$ if and only if the 1-dimensional
component of $\mathcal{Z}$ is parametrizable over $\K$.
\end{itemize}
\end{theorem}
\begin{proof}
\cite{ARS-2,Ultraquadrics,Tesis-tab}
\end{proof}

So, the problem of parametrizing $\C$ over $\K$ can be translated to the problem
of parametrizing the one-dimensional component of the witness variety over $\K$.
The interesting thing is that this problem is related with the study of
hypercircles:

\begin{defn}\label{def:hypercircle}
Let $\frac{at+b}{ct+d}\in \K(\alpha)(t)$ represents a $\K(\alpha)$-isomorphism
of $\F(t)$, $a,b,c,d\in \K(\alpha)$, $ad-bc\neq 0$. Write
\[\frac{at+b}{ct+d}=\psi_0(t)+\alpha \psi_1(t)+ \cdots+ \psi_{n-1}(t)\]
where $\psi_i(t)\in \K(t)$. The \textit{$\alpha$-hypercircle} associated to
$\frac{at+b}{ct+d}$ for the extension $\K\subseteq \K(\alpha)$ is the parametric
curve in $\F^n$ given by the parametrization $(\psi_0,\ldots, \psi_{n-1})$.
\end{defn}

Usually, if the extension $\K\subseteq \K(\alpha)$ and the primitive element
$\alpha$ are assumed, we could just talk about the witness variety of $\phi$ or
the hypercircle associated to $\frac{at+b}{ct+d}$. But, since we will
explicitly change the algebraic extension to consider or the primitive element,
we will try to avoid this language as it can be misleading.

It can be easily proved that fixed $\K\subseteq \K(\alpha)$ and a primitive
element $\alpha$, two linear fraction $u=\frac{at+b}{ct+d}$,
$v=\frac{a't+b'}{c't+d'}\in \K(\alpha)(t)$ define the same hypercircle if and
only there is a linear fraction $w=\frac{a''t+b''}{c''t+d''}\in \K(t)$ such that
$u(t)=v(w(t))$.

Note also that if $c=0$ in the linear fraction, then the hypercircle is a line
parametrizable over $\K$. Anyway, by composing on the right by an appropriate
linear fraction $w\in \K(t)$, we can always suppose that the associated unit of
the hypercircle is of the form $\frac{at+b}{t+d}$ ($c=1$). We will suppose that
this is the case from now on unless otherwise specified.

Now we show some other useful facts about hypercircles. The proofs of these
facts can be checked in \cite{Generalizing-circles}

\begin{lemma}\label{lem:lema_grado_d}
Let $\mathcal{U}$ be the $\alpha$-hypercircle associated to $\frac{at+b}{t+d}\in
\K(\alpha)(t)$ for the extension $\K\subseteq \K(\alpha)$. Then, the geometric
degree of $\mathcal{U}$ equals $[\K(d) :\K]$
\end{lemma}

\begin{defn}
Let $\mathcal{U}$ be an $\alpha$-hypercircle, if the geometric degree of
$\mathcal{U}$ equals $n=[\K(\alpha) :\K]$ then $\mathcal{U}$ is called a
\textit{primitive} hypercircle. i.e. If the $d$ of an associated unit
$\frac{at+b}{t+d}$ is a primitive element for the extension $\K\subseteq
\K(\alpha)$.
\end{defn}

\begin{lemma}\label{lem:forma-reducida}
Let $u(t)=\frac{at+b}{t+d}$ be a unit and $\mathcal{U}$ its associated
$\alpha$-hypercircle for the extension $\K\subseteq \K(\alpha)$. Then
$\mathcal{U}$ is affinely equivalent over $\mathbb{K}$ to the
$\alpha$-hypercircle $\mathcal{U}^\star$ generated by
$u^{\star}(t)=\frac{1}{t+d}$ for the extension $\K\subseteq \K(\alpha)$.
\end{lemma}

\begin{theorem}\label{teo:paso_a_propio}
Let $\mathcal{U}$ be the non-primitive $\alpha$-hypercircle associated to
$u(t)=\frac{at+b}{t+d}\in\mathbb{K}(\alpha)(t)$ for the extension $\K\subseteq
\K(\alpha)$. Let $\mathcal{V}$ be the $d$-hypercircle associated to the unit
$\frac{1}{t+d}$ for the extension $\mathbb{K}\subseteq\mathbb{K}(d)$. Then,
there is an affine inclusion from $\mathbb{F}^r$ to $\mathbb{F}^n$, defined over
$\mathbb{K}$, that maps $\mathcal{V}$ onto $\mathcal{U}$.
\end{theorem}

\begin{proposition}\label{prop:puntos-infinito-directo}
Let $M_{\alpha}(t)$ be the minimal polynomial of $\alpha$ over $\mathbb{K}$. Let
$m_{\alpha}(t)= \frac{M_{\alpha}(t)}{t-\alpha}=
\sum_{i=0}^{n-1}l_it^i\in\mathbb{K}(\alpha)[t]$, where $l_{n-1}=1$. Then, the
points at infinity of every primitive $\alpha$-hypercircle for the extension
$\K\subseteq \K(\alpha)$ are
\[[l_0: l_1: \cdots: l_{n-2}: l_{n-1}: 0]\] and its conjugates.
\end{proposition}

With the notion of hypercircle we can extend
Theorem~\ref{teo:witness_variety_properties} by the following:

\begin{theorem}\label{teo:witness_and_hypercircles}
In the previous conditions, $\K$ is a field of parametrization of $\C$ and
$\phi(\frac{at+b}{ct+d})$ is a parametrization of $\C$ over $\K$ if and only if
the one-dimensional component of the witness variety $\mathcal{Z}$ is an
$\alpha$-hypercircle associated to $\frac{at+b}{ct+d}$ for the extension
$\K\subseteq\K(\alpha)$.
\end{theorem}
\begin{proof}
\cite{ARS-2,Ultraquadrics}
\end{proof}

This Theorem will be essential in the discussion of the reparametrization
algorithm, since we will use different reparametrization units, different
witness varieties and, specially, different extensions.

As a Consequence of Corollary~\ref{cor:fields_of_parametrization} we can prove
that any witness variety is indeed a hypercircle but possibly with respect to a
different algebraic extension.

\begin{theorem}\label{teo:vteshc}
Let $\mathcal{C}$ be a curve $\mathbb{K}$-definable, that is not
$\mathbb{K}$-parametrizable, where
$\mathbb{K}=\mathbb{Q}(y_1,\ldots,y_u)(\beta)$ as in Section~\ref{sec:prelim}
such that $\mathcal{C}$ is given by a $\mathbb{K}(\alpha)$-parametrization. Let
$\mathcal{U}$ be the 1-dimensional component of the $\alpha$-witness variety of
$\mathcal{C}$ for the extension $\K\subseteq \K(\alpha)$. Then, there are
infinitely many $\eta$, quadratic elements over $\mathbb{K}$, such that
$\mathcal{U}$ is an $\alpha$-hypercircle with respect to the algebraic extension
$\mathbb{K}(\eta) \subseteq \mathbb{K}(\eta, \alpha)$.
\end{theorem}
\begin{proof}
By Corollary~\ref{cor:fields_of_parametrization}, there are infinitely many
distinct quadratic fields of parametrization $\mathbb{K}(\gamma)$ of
$\mathcal{C}$. Let $\eta$ be any of this quadratic element over $\mathbb{K}$
such that $\eta$ does not belong to the normal closure of $\mathbb{K}(\alpha)$
over $\mathbb{K}$. In particular, there are infinitely many $\eta$ satisfying
this condition. The minimal polynomial of $\alpha$ over $\mathbb{K}(\eta)$
equals its minimal polynomial over $\mathbb{K}$. Hence, by the computational
definition of the witness variety, the output is the same considering the
extension $\K\subseteq \K(\alpha)$ or $\K(\eta)\subseteq \K(\eta, \alpha)$. But,
since $\mathcal{C}$ can be parametrized over $\mathbb{K}(\eta)$, from
Theorem~\ref{teo:witness_variety_properties}, it follows that $\mathcal{U}$ is
an
$\alpha$-hypercircle with respect to the extension $\mathbb{K}(\eta) \subseteq
\mathbb{K}(\eta, \alpha)$.
\end{proof}

This result is important because hypercircles have a well known geometry. So, in
particular, we automatically know many geometric properties of witness
varieties, including the possible degrees, the Hilbert function, the smoothness,
the structure at infinity or specific algorithms for parametrization and
implicitization (See \cite{Generalizing-circles}).

\section{Optimal Affine Reparametrization of a Curve}
We have seen that any witness variety is a hypercircle for a convenient
algebraic extension. In this Section, we present a method of optimal
reparametrization of a curve by affine change of variables only. We will always
suppose that our base field $\K$ is as in Section~\ref{sec:prelim}. Suppose that
$\mathcal{C}$ is given by a parametrization $\phi$ over $\mathbb{K}(\alpha)$. We
want to obtain reparametrizations of $\mathcal{C}$ by affine change of variables
$t\mapsto e_1t+e_2$ only. In this case, there is a minimum field (up to
$\K$-isomorphism) $\mathbb{K}(\gamma)$ such that $\phi(e_1 t+e_2) \in
\mathbb{K}(\gamma)(t)$. That is, there are $e_1, e_2$ such that
$\phi(e_1t+e_2)\in \mathbb{K}(\gamma)(t)$ and, for every pair $e_1', e_2'\in
\mathbb{F}$, $e_1'\neq 0$, the field generated over $\mathbb{K}$ by the
coefficients of $\phi(e_1't+e_2')$ contains (a field isomorphic to)
$\mathbb{K}(\gamma)$.  Moreover, to obtain a reparametrization over
$\mathbb{K}(\gamma)$, we do not need any point over the ground field $\K$. This
is a generalization of the reparametrization problem for polynomially
parametrizable curves in \cite{SV-Polynomial}. We now proceed in several steps
towards the algorithm.

\begin{proposition}\label{prop:buen_punto_infinito}
Let $\mathcal{C}$ be a $\mathbb{K}$-definable curve given by a parametrization
over $\mathbb{K}(\alpha)$. Let $\mathcal{U}$ be the 1 dimensional component of
the $\alpha$-witness variety of $\mathcal{C}$. Then, there is at least one point
at infinity of $\mathcal{U}$ that admits a representation over
$\mathbb{K}(\alpha)$.
\end{proposition}
\begin{proof}
Suppose first that $\mathcal{C}$ can be parametrizable over $\K$. Then, by
Theorem~\ref{teo:witness_variety_properties} $\mathcal{U}$ is an
$\alpha$-hypercircle for the extension $\K\subseteq \K(\alpha)$ associated to a
unit $\frac{at+b}{t+d}\in \K(\alpha)(t)$. If $\mathcal{U}$ is a primitive
$\alpha$-hypercircle,
we have explicitly presented a point at infinity with a representation over
$\K(\alpha)$ in
Proposition~\ref{prop:puntos-infinito-directo}. If $\mathcal{U}$ is not a
primitive $\alpha$-hypercircle, then, by Theorem~\ref{teo:paso_a_propio},
$\mathcal{U}$ is the image of a primitive $d$-hypercircle $\mathcal{U}_1$ for
the extension $\mathbb{K}\subseteq \mathbb{K}(d)$ under an affine inclusion
$\F^r\rightarrow \F^n$ defined over $\K$, $r=[\K(d): \K]$. Hence, at least one
point at infinity of $\mathcal{U}_1$ has a representation over
$\mathbb{K}(d)\subseteq \mathbb{K}(\alpha)$. As the affine inclusion is defined
over $\mathbb{K}$, the corresponding point at infinity of $\mathcal{U}$ also
admits a representation in $\mathbb{K}(\alpha)$.

Suppose now that $\mathcal{C}$ is not $\mathbb{K}$-parametrizable. Then, by
Theorem~\ref{teo:vteshc}, there are infinitely many quadratic elements $\eta$
over $\mathbb{K}$ such that $\mathcal{U}$ is an $\alpha$-hypercircle with
respect to the extension of fields $\mathbb{K}(\eta) \subseteq \mathbb{K}(\eta,
\alpha)$. As there are only finite points at infinity, we conclude that there is
a point $p$ at infinity that admits a representation over infinitely many fields
of the form $\mathbb{K}(\eta, \alpha)$. Necessarily, this point admits a
representation over $\mathbb{K}(\alpha)$.
\end{proof}

Let $\mathcal{C}$ be a parametric curve parametrizable over $\mathbb{K}$ but
such that it is given by a parametrization over $\mathbb{K}(\alpha)$.
$[\mathbb{K}(\alpha) :\mathbb{K}]=n$. Suppose that the associated hypercircle
$\mathcal{U}$ to $\mathcal{C}$ is of degree $r< n$. Let $u(t)= \frac{at+b}{t+d}
\in \mathbb{K}( \alpha)(t)$ be a unit associated to $\mathcal{U}$. Then
$\mathbb{K}(d) \subsetneq \mathbb{K}(\alpha)$ and $[\mathbb{K}(d) :
\mathbb{K}]=r$. By Theorem~\ref{teo:paso_a_propio}, $\mathcal{U}$ is
$\mathbb{K}$-affinely equivalent to the primitive $d$-hypercircle defined by
$\frac{1}{t+d}$ for the extension $\K\subseteq \K(d)$ in $\mathbb{K}^r$. Here,
we present how to compute $\K(d)$ from the implicit equations of $\mathcal{U}$
and a reparametrization of $\mathcal{C}$ over $\mathbb{K}(d)$.

\begin{proposition}\label{prop:cuerpo_intermedio}
If $\mathcal{U}$ is the $\alpha$-hypercircle associated to $\frac{at+b}{t+d}$
for the extension $\K\subseteq \K(\alpha)$, let $[a_0: \ldots: a_{n-1}: 0 ]$ be
a point at infinity of $\mathcal{U}$ given by a representation over
$\mathbb{K}(\alpha)$, suppose that it is dehomogenized with respect to an index
$i$. Then, $\mathbb{K}(d)$ is $\K$-isomorphic to $\mathbb{K}(a_0, \ldots,
a_{n-1})$. In particular, there is another point at infinity $[b_0: \ldots:
b_{n-1}: 0 ]$ of $\mathcal{U}$, given by a representation over $\K(\alpha)$,
dehomogenized with respect to an index $i$ such that
$\K(d)=\K(b_0,\ldots,b_{n-1})$.
\end{proposition}
\begin{proof}
By Lemma~\ref{lem:forma-reducida} $\mathcal{U}$ is affinely equivalent over
$\mathbb{K}$ to the $\alpha$-hypercircle $\mathcal{U}_1$ associated to
$\frac{1}{t+d}$
for the extension $\K\subseteq \K(\alpha)$. Thus, the (dehomogenized) points at
infinity of $\mathcal{U}$ and $\mathcal{U}_1$ generate the same algebraic
extension over $\mathbb{K}$. So, without loss of generality, we may suppose that
$u(t)=\frac{1}{t+d}$ is the unit associated to $\mathcal{U}$. Let
$M(t)=t^r+k_{r-1}t^{r-1}+\cdots+k_0$ be the minimal polynomial of $-d$ over
$\mathbb{K}$ and let \[m(t)= \frac{M(t)}{t+d} =l_{r-1}t^{r-1} +l_{r-2} t^{r-2}+
\cdots+ l_0 \in \mathbb{K}(d)[t].\] Let $\mathcal{U}_d \subseteq \mathbb{K}^r$ be
the $d$-hypercircle associated to $u(t)$ for the extension of fields $\mathbb{K}
\subseteq \mathbb{K}(d)$. By Proposition~\ref{prop:puntos-infinito-directo} the
points at infinity of $\mathcal{U}_d$ are
$[l_0: l_1: \cdots: l_{r-2}: l_{r-1}: 0]$
and its conjugates. Notice that $l_{r-2}=k_{r-1}-d$, so $\mathbb{K}( l_0,
\ldots, l_{r}) = \mathbb{K}(d)$. Finally, from Theorem~\ref{teo:paso_a_propio}
since there is an affine inclusion $\mathbb{F}^r \rightarrow \mathbb{F}^n$ that
maps $\mathcal{U}_d$ onto $\mathcal{U}_1$ that is defined over $\mathbb{K}$, the
field that generates the points at infinity is the same. In particular, the
point $[l_0: l_1: \cdots: l_{r-2}: l_{r-1}: 0]$ is mapped to a point
$[b_0: \ldots: b_{n-1}: 0]$ that has a representation over $\K(d)\subseteq
\K(\alpha)$. By conjugation, $\mathbb{K}(a_0, \ldots, a_{n-1})$ is isomorphic to
$\K(b_0, \ldots, b_{n-1})= \mathbb{K}(l_0, \ldots, l_{n-1})= \mathbb{K}(d)$.
\end{proof}

Once we know how to compute $\K(d)$, we have a method to reparametrize a curve
over
$\mathbb{K}(d)$.

\begin{theorem}\label{teo:aplicacion_rep_afin}
Let $\mathcal{C}$ be a curve $\mathbb{K}$-definable given by a birational
parametrization $\phi$ with coefficients in $\mathbb{K}(\alpha)$. Let
$\mathcal{U}$ be the 1-dimensional component of the $\alpha$-witness variety of
$\mathcal{C}$ for the extension $\K\subseteq\K(\alpha)$. Suppose that the degree
of $\mathcal{U}$ is $r<n$. Then, $\mathcal{C}$ admits a reparametrization over
$\mathbb{K}(\gamma)\subseteq \mathbb{K}(\alpha)$, where $[\mathbb{K}(\gamma):
\mathbb{K}]=r$. This reparametrization can be taken affine, $t\rightarrow
f_1t+f_2\in \K(\alpha)[t]$.

Moreover, if $e_1, e_2 \in \mathbb{F}$, $e_1\neq 0$ are algebraic, let $\phi(e_1
t+e_2)$ be another parametrization of $\mathcal{C}$ and let $\mathbb{L}$ be the
field generated over $\mathbb{K}$ by the coefficients of $\phi(e_1t+e_2)$, then
\begin{enumerate}
\item $[\mathbb{L}:\mathbb{K}]\geq r$.
\item $\mathbb{L}$ contains (a field isomorphic to) $\mathbb{K}(\gamma)$.
\item If $[\mathbb{L}: \mathbb{K}]=r$ then $\mathbb{L}$ is isomorphic to
$\mathbb{K}(\gamma)$.
\item There are $e_1', e_2'\in \mathbb{L}$ such that $e_1't+e_2'$ reparametrizes
$\phi$ over (a field isomorphic to) $\mathbb{K}(\gamma)$.
\end{enumerate}
\end{theorem}
\begin{proof}
If $\mathcal{C}$ is not parametrizable over $\K$, by Theorem~\ref{teo:vteshc},
there is a field $\K(\eta)$ quadratic over $\K$ such that $\mathcal{U}$ is an
$\alpha$-hypercircle for the extension $\mathbb{K}(\eta) \subseteq
\mathbb{K}(\eta, \alpha)$ and the intersection of $\K(\eta)$ and
the normal closure of $\K(\alpha)$ is $\K$. If $\mathcal{C}$ is parametrizable
over $\K$, then it is trivially true that there exists such a field $\K(\eta)$.

Since $\mathcal{U}$ is an $\alpha$-hypercircle for the extension
$\mathbb{K}(\eta) \subseteq \mathbb{K}(\eta, \alpha)$, there is an associated
unit $u(t)=\frac{at+b}{t+d}\in \mathbb{K}(\eta, \alpha)(t)$ so that
$\phi(u(t))\in \K(\eta)(t)$. By Proposition~\ref{prop:cuerpo_intermedio}, there
is a point at infinity $[b_0: \ldots: b_{n-1}: 0]$ of $\mathcal{U}$, $b_i\in
\K(\alpha)$, $0\leq i\leq n-1$ such that $\mathbb{K}(\eta, d) =\mathbb{K}(\eta,
b_0, \ldots, b_{n-1})$. Let $\gamma$ be a primitive element for the extension
$\K\subseteq \K(b_0,\ldots, b_{n-1})$, $\gamma\in \K(\alpha)$. We have that
$[\K(\eta, d): \K(\eta)]=r$, so $[\K(\eta, \gamma) :\K(\eta)]=r$ and, since
$\K(\eta)$ intersects the normal closure of $\K(\alpha)$ trivially, we have that
$[\K(\gamma): \K]=r$.

Let us perform Weil's construction again but now with respect to the extension
$\K(b_0,\ldots,b_{n-1})\subseteq \K(\alpha)$ with primitive element $\alpha$.
Since the minimal polynomial of $\alpha$ over $\K(\gamma)$ equals the minimal
polynomial of $\alpha$ over $\K(\eta, \gamma)$, we get that the witness variety
$\mathcal{U}'$ of $\mathcal{C}$ for the extension $\K(\gamma)\subseteq
\K(\alpha)$ equals the witness variety of $\mathcal{C}$ for the extension
$\K(\eta, \gamma)=\K(\eta, d)\subseteq \K(\eta, \alpha)$.

Now $\frac{at+b}{t+d}$ reparametrizes $\phi$ over $\K(\eta)$, so $\frac{at +b}{t
+d}$ is also an associate unit of $\mathcal{U}'$ for the extension $\K(\eta,
d)\subseteq \K(\eta, \alpha)$.
Since $1/t-d\in \K(\eta, d)(t)$, we know that
\[\frac{at+b}{t+d}\circ \left(\frac{1}{t} -d\right)=a+t(b-ad)\]
is also a linear fraction associated to $\mathcal{U}'$ (for the extension
$\K(\eta, \gamma)\subseteq \K(\eta, \alpha)$ with primitive element $\alpha$).
In this case the denominator is $1$, so $\mathcal{U}'$ must be a line.

To sum up, $\mathcal{U}'$, the witness variety of $\mathcal{C}$ for the
extension $\K(\gamma) \subseteq \K(\alpha)$ with primitive element $\alpha$ is a
line, defined over $\K(\gamma)$, so it has an associated unit of the form
$f_1t+f_2\in \K(\alpha)(t)$ such that $\phi(f_1t+f_2)\in \K(\gamma)(t)$.

For the second part, let $\eta$ be a quadratic element in the conditions of
Theorem~\ref{teo:vteshc} such that it does not belong to the normal closure of
$\mathbb{L}(\alpha)$ over $\mathbb{K}$. Let $u=\frac{at+b}{t+d}\in \K(\eta,
\alpha)(t)$ be a unit that
reparametrizes $\mathcal{C}$ over $\mathbb{K}(\eta)$. Let $\phi_e=
\phi(e_1t+e_2)\in \mathbb{L}(t)$. On the one hand,
\[w_1(t)=\frac{\frac{a-e_2}{e_1}t
+\frac{b-e_2d}{e_1}}{t+d}=\left(\frac{t-e_2}{e_1}\right)\circ
\frac{at+b}{t+d}\in \mathbb{L}(\eta, \alpha, e_1, e_2)(t)\]
reparametrizes $\phi_e$ over $\mathbb{K}(\eta)$. On the other hand, by
Theorem~\ref{teo:witness_variety_properties}, there is another unit
$w_2=\frac{a't+b'}{t+d'}\in \mathbb{L}(\eta)(t)$ that reparametrizes $\phi_e$
over $\mathbb{K}(\eta)$. Then, there is a unit $w_3=\frac{a''t+b''}{t+d''}\in
\mathbb{K}(\eta)(t)$ such that $w_1=w_2\circ w_3\in \mathbb{L}(\eta)(t)$. Hence,
$d\in \mathbb{L}(\eta)$. So
\[[\L: \K]=[\L(\eta): \K(\eta)]= [\K(\eta, d): \K(\eta)][\L(\eta):\K(\eta, d)].\]
Since $[\K(\eta, d): \K(\eta)]=r$, $[\L :\K] \geq r$ and we have the first
item.

Now, $\K(\eta,d)$ is $\K(\eta)$-isomorphic to $\K(\eta, \gamma)$. Then
$\L(\eta)$ contains a field that is $\K(\eta)$-isomorphic to $\K(\eta, \gamma)$.
Let $M(x)$ be the minimal polynomial of $\gamma$ over $\K$. Then, there is a
root of $M(x)$ in $\L(\eta)$. By the election of $\eta$, this root belongs to
$\L$. So $\L$ contains a field $\K$-isomorphic to $\K(\gamma)$. The rest of the
items follows easily from this one and the proof of the first part.
\end{proof}

As a corollary of this Theorem, we have the following algorithm to compute an
optimal affine reparametrization.

\begin{algo}\mbox{}\\
{\bf input:} A computable field $\mathbb{K}$ (with factorization) finitely
generated over $\mathbb{Q}$. An element $\alpha$ algebraic of degree $n$ over
$\mathbb{K}$. A birational parametrization $\phi(t)\in \mathbb{K}(t)^N$ of a
rational curve
$\mathcal{C}$.\mbox{}\\
{\bf output:} \begin{itemize}\item[--] If $\mathcal{C}$ is not defined over
$\mathbb{K}$ return ``FAIL".
\item[--] If $\mathcal{C}$ is defined over $\K$ return $at+b\in \K(\alpha)[t]$
such that \[[\K(\textrm{coeffs}(\phi(at+b))): \K]\] is the smallest possible.
\end{itemize}
\begin{enumerate}
\item Compute the witness variety $\mathcal{Z}$ of $\mathcal{C}$ for the
extension $[\K(\alpha): \K]$ and primitive element $\alpha$.
\item Compute the points at infinity of $\mathcal{Z}$.
\item If there are no points at infinity ($\mathcal{Z}$ is zero dimensional)
then {\bf Return} ``FAIL".
\item Else, (there are finitely many points at infinity).\\
      Compute a point $[b_0: \ldots: b_{n-1}: 0]$ at infinity with coefficients
in
$\K(\alpha)$.
\item Compute $r=[\K(b_0, \ldots, b_{n-1}) :\K]$.
\item If $r=[\K(\alpha): \K]$, {\bf Return} $at+b=t$.
\item If $r=1$
\begin{itemize}
\item Compute $(\psi_0,\ldots,\psi_{n-1})$ a linear parametrization of the
unique one-dimensional component of $\mathcal{Z}$ over $\mathbb{K}$ (easy, since
it is a line defined over $\K$).
\item {\bf Return} $at+b=\sum_{i=0}^{n-1} \psi_i(t)\alpha^i$.
\end{itemize}
\item (Else) Compute the minimal polynomial of $\alpha$ over
$\K(b_0,\ldots,b_{n-1})$.
\item Compute the witness variety $\mathcal{Z}'$ of $\mathcal{C}$ for the
extension $\K(b_0,\ldots,b_{n-1})\subseteq \K(\alpha)$.
\item Compute $(\psi_0,\ldots, \psi_{n/r-1})$ a linear parametrization of the
unique one-dimensional component of $\mathcal{Z}'$ (which is a line) with
coefficients in the field $\K(b_0, \ldots, b_{n-1})$.
\item {\bf Return} $at+b=\sum_{i=0}^{n/r-1}\psi_i(t)\alpha^i$.
\end{enumerate}

\end{algo}

\begin{example}
Let $\alpha$ be a root of $x^4-4x^3+12x^2-16x+8$, and let $\mathcal{C}$ be the
parametric curve given by
\[\begin{array}{l} \displaystyle{x=\frac{-6 +18\alpha -9\alpha^2 +6\alpha^3
+(44 -52\alpha +18\alpha^2 -4\alpha^3)t -4t^2}{ -22  +26\alpha -9\alpha^2
+2\alpha^3 +4t}},\\[1.5em]
y=\displaystyle{\frac{ -12 -2\alpha +9\alpha^2 -\alpha^3 +\left( 4 +4\alpha
+4\alpha^2 \right)t +\left( 12 -16\alpha +6\alpha^2 -2\alpha^3\right)t^2}{-22
+26\alpha -9\alpha^2 +2\alpha^3 +4t}}
\end{array}\]
The hypercircle $\mathcal{U}$ associated to this curve has implicit equations:
\[\{4t_2 +12t_3 -3, 5 +2t_{1} -16t_3, 2t_0^2 +24t_3t_0 +80t_3^2 -10t_0 -52t_3
+15\}.\]
One can easily check that this hypercircle is non primitive, because it is
contained in the hyperplane $4t_2+12t_3-3$. Moreover, from its equations, it is
a conic. The points at infinity are: \[[2\gamma: 8: -3: 1:0]\] where $\gamma$ is
a
root of $x^2+6x+10$. The roots of this polynomial in $\mathbb{Q}(\alpha)$ are
$-4\alpha+3/2\alpha^2-1/2\alpha^3$ and $-6+4\alpha-3/2\alpha^2+1/2\alpha^3$.
Choose for example $\gamma=-4\alpha+3/2\alpha^2-1/2\alpha^3$. Then, the minimal
polynomial of $\alpha$ over $\mathbb{Q}(\gamma)$ is
$x^2+(-8-2\gamma)x+8+2\gamma$. Now,
we rewrite the parametrization of $\mathcal{C}$ over this extension of fields:
\[\begin{array}{l} x=\displaystyle{\frac{(21 +9\gamma) \alpha -39
-15\gamma+((6\gamma +14) \alpha -2\gamma -2) t-2t^2}{1 +\gamma+( -3\gamma
-7)\alpha +2t}}\\[1.5em]
y=\displaystyle{\frac{-30 -5\gamma +(6\gamma +27)\alpha+( -14 -4\gamma +(18
+4\gamma) \alpha)t +(2\gamma +6)t^2}{1 +\gamma+( -3\gamma -7)\alpha +2t}}
\end{array} \]
we compute the hypercircle associated to the extension $\mathbb{Q}(\gamma)
\subseteq
\mathbb{Q}(\gamma, \alpha)$. We know that it will be a line, in fact, the
computation
yields $2t_1-3\gamma-7$, that can be parametrized by $(s,(3\gamma+7)/2)$. Hence,
the affine substitution $t=t+(3\gamma+7)/2\alpha$ in the parametrization yields
a parametrization over the subfield $\mathbb{Q}(\gamma)$
\[\begin{array}{l}
x=\displaystyle{\frac{-3\gamma-2t^2\gamma+4t\gamma-5+6t^2-4t^3}{5-8t+4t^2}}\\[
1.5em]
y=\displaystyle{2\frac{7t\gamma+2t^3\gamma-3\gamma-6t^2\gamma-10+23t+6t^3-19t^2}
{5-8t+4t^2}} \end{array} \]
\end{example}

\section*{Acknowledgements}
The author wants to thank Tomas Recio, Rafael Sendra and Carlos Villarino for useful discussion and sugestions. And Erwan Brugall\'e for helping correcting a typo about integral extensions.

\noindent
Luis Felipe Tabera\\
IMDEA Mathematics Foundation\\
M\'odulo C-IX 3th floor Facultad de Ciencias - UAM\\
Ciudad Universitaria de Cantoblanco\\
E-28049 Madrid, Spain.\\
e-mail: luis.tabera@imdea.org\\
http://personales.unican.es/taberalf/

\end{document}